\documentclass[smallextended]{svjour3-ppn}
\smartqed
\usepackage{graphicx}
\journalname{Preprint}
\usepackage{amsmath,amsfonts,amssymb,amsxtra,setspace,xspace}
\usepackage[colorlinks=true]{hyperref}
\hypersetup{urlcolor=blue, citecolor=red, linkcolor=blue}
\newtheorem{thm}{Theorem}

\newtheorem{lem}[thm]{Lemma}
\newtheorem{prop}[thm]{Proposition}

\newcommand{\R}{{\mathbb R}}

\newcommand{\be}[1]{\begin{equation}\label{#1}}
\newcommand{\ee}{\end{equation}}
\renewcommand{\(}{\left(}
\renewcommand{\)}{\right)}
\newcommand{\ird}[1]{\int_{\R^d}{#1}\,dx}
\newcommand{\nrm}[2]{\|{#1}\|_{\mathrm L^{#2}(\R^d)}}

\usepackage{color}

\begin{document}

\title{Nonlinear diffusions: extremal properties of Barenblatt profiles, best matching and delays}
\titlerunning{Extremal properties of Barenblatt profiles and delays}
\author{Jean Dolbeault \and Giuseppe Toscani}
\institute{J. Dolbeault \at Ceremade (UMR CNRS no. 7534), Universit\'e Paris-Dauphine, Place de Lattre de Tassigny, F-75775 Paris C\'edex 16, France\\
\email{dolbeaul@ceremade.dauphine.fr}
\and G. Toscani \at University of Pavia Department of Mathematics, Via Ferrata~1, 27100 Pavia, Italy \\
\email{giuseppe.toscani@unipv.it}}
\date{January 9, 2015}
\maketitle
\begin{abstract} In this paper, we consider functionals based on moments and nonlinear entropies which have a linear growth in time in case of source-type solutions to the fast diffusion or porous medium equations, that are also known as Barenblatt solutions. As functions of time, these functionals have convexity properties for generic solutions, so that their asymptotic slopes are extremal for Barenblatt profiles. The method relies on scaling properties of the evolution equations and provides a simple and direct proof of sharp Gagliardo-Nirenberg-Sobolev inequalities in scale invariant form. The method also gives refined estimates of the growth of the second moment and, as a consequence, establishes the monotonicity of the delay corresponding to the best matching Barenblatt solution compared to the Barenblatt solution with same initial second moment. Here the notion of best matching is defined in terms of a relative entropy.
\keywords{Nonlinear diffusion equations \and Source-type solutions \and Gagliardo-Nirenberg-Sobolev inequalities \and Improved inequalities \and Scalings \and Second moment \and Temperature \and R\'enyi entropy \and Best matching Barenblatt profiles \and Delay}
\def\subclassname{{\bfseries Mathematics Subject Classification (2010)}\enspace}
\subclass{primary: 35K55 \and 35K65 \and 35B40;
secondary: 46E35 \and 39B62 \and 49J40}
\end{abstract}
\section{Introduction}\label{Sec:Intro}

Consider the nonlinear diffusion equation in $\R^d$, $d\ge1$,
\be{poro}
\frac{\partial u}{\partial t}=\Delta u^p\,,
\ee
with initial datum $u(x,t=0)=u_0(x)\ge0$ such that
\be{second}
\ird{u_0}=1\,,\quad\Theta(0)=\frac1d\ird{|x|^2\,u_0}<+\infty\,.
\ee
It is known since the work of A.~Friedman and S.~Kamin,~\cite{MR586735}, that the large time behavior of the solutions to~\eqref{poro} is governed by the source-type Barenblatt solutions
\be{ba-self}
\mathcal U_\star(t,x):=\frac1{\big(\kappa\,t^{1/\mu}\big)^d}\,\mathcal B_\star\!\(\frac x{\kappa\, t^{1/\mu}}\)
\ee
where
\[
\mu:=2+d\,(p-1)\,,\quad\kappa:=\Big|\frac{2\,\mu\,p}{p-1}\Big|^{1/\mu}
\]
and $\mathcal B_\star$ is the Barenblatt profile
\[
\mathcal B_\star(x):=
\begin{cases}
\big(C_\star-|x|^2\big)_+^{1/(p-1)}&\text{if}\;p>1\,,\\[5pt]
\big(C_\star+|x|^2\big)^{1/(p-1)}&\text{if}\;p<1\,.
\end{cases}
\label{ba}
\]
Here $(s)_+=\max\{s,0\}$ and, for $p>1-\frac2d$, the constant $C_\star$ is chosen so that $\ird{\mathcal B_\star(x)}=1$. Solution~\eqref{ba-self} was found around 1950 by Ya.B.~Zel'dovich and A.S.~Kompaneets. Later G.I.~Barenblatt analyzed the solution representing heat release from a point source for $p>1$. See~\cite{zel1950towards,barenblatt1952some}, and~\cite{MR2286292} for more~details.

The result of~\cite{MR586735}, subsequently improved in a number of papers like, \emph{e.g.},~\cite{MR2909910}, guarantees to various extents that the Barenblatt solution can fruitfully be used to obtain the large-time behavior of all other solutions, departing from an initial value like in~\eqref{second}. Precise estimates of the difference in time of basic quantities of the real solution with respect to the source-type one are however difficult to obtain and additional conditions are needed for low values of $p>0$.

As main example, let us consider the growth of the second moments of the solution to~\eqref{poro}, for an initial value satisfying~\eqref{second}. For a Barenblatt solution, we easily obtain the exact growth
\[
\frac 1d\ird{|x|^2\,\mathcal U_\star(t,x)}=t^{2/\mu}\,\Theta_\star\quad\mbox{where}\quad\Theta_\star:=\frac{\kappa^2}d\ird{|x|^2\,\mathcal B_\star(x)}\,.
\]
For $p>1$ it has been shown in~\cite{MR2133441} that this time-behaviour is captured by any other solution to~\eqref{poro} satisfying conditions~\eqref{second} for large times. Further studies published in~\cite{2014arXiv1408.6781D} then revealed that, at least in the fast diffusion range, there is a nondecreasing \emph{delay} between the propagation in time of the second moments of the source-type solution and the generic solution to \eqref{poro}. We shall give a simple explanation for this fact in Theorem~\ref{Thm:T3}, and also show that in the porous medium case $p>1$ the delay is nonincreasing.

A way to look at this phenomenon is the following. The second moment of the Barenblatt solution, raised to power $\mu/2$, grows linearly in time. Hence, its first variation in time of $\Theta_\star(t)^{\mu/2}$ is constant, while the second variation is equal to zero. To see how the second moment of any other solution to equation~\eqref{poro} behaves in time, it comes natural to estimate the time variations of $\mathsf G:=\Theta^{\mu/2}$, where the \emph{second moment functional} is defined by
\[
\Theta(t):=\frac 1d\ird{|x|^2\,u(t,x)}\,.
\]
This analysis will lead to the interesting observation that the second variation of $\mathsf G$ has a fixed sign, which implies concavity or convexity in $t$ depending on the value of $p$. Also, since the first variation is invariant with respect to dilations and the solution to~\eqref{poro} is asymptotically self-similar, it follows at once that the first variation satisfies a sharp inequality which connects the second moment to its first variation, and optimality is achieved by Barenblatt profiles.

This idea can also be applied to the study of the behavior of the derivative of the second moment, as predicted by the nonlinear diffusion. It holds
\[\label{fi-va}
\Theta'=2\,\mathsf E
\]
where, up to a sign, the \emph{generalized entropy functional} is defined by
\[
\mathsf E(t):=\ird{|u(t,x)|^p}\,.
\]
For the Barenblatt solution~\eqref{ba-self} one can easily check that $\mathsf E_\star(t)^\sigma$ grows linearly in time if $d\,(1-p)\,\sigma=\mu$. This leads to estimate the time variation of $\mathsf F:=\mathsf E^\sigma$. As in the previous case, we will conclude by showing that the second variation in time of $\mathsf F$ has a negative sign if $p<1$, which implies concavity. In this case also, the first variation is invariant with respect to dilations, which implies a sharp inequality of Gagliardo-Nirenberg-Sobolev type, and optimality is again achieved by Barenblatt profiles.

As the results of this paper clearly indicate, looking at functionals of the solution to the nonlinear diffusion\eqref{poro} which grow linearly in time when evaluated along the source-type solution represent a valid and powerful alternative to the well-known \emph{entropy-entropy production method} developed in the last two decades to investigate the large-time behavior of the solution by scaling the problem in order to obtain a Fokker-Planck equation with a fixed steady state.

This new way of looking at the problem opens completely new questions, which at the moment seem to be very difficult to deal with. For instance, higher order derivatives in time of the second moment of the solution are easily evaluated in correspondence to Barenblatt solutions, by giving a precise growth, say $t^{-\sigma_n}$, where $n$ is the order of the derivative. Hence, one can think to proceed as for the first two cases, by considering the power of order $\sigma_n$ of the $n$-th derivative and by checking if there is a sign.

\medskip With Theorem~\ref{Thm:T1}, we shall start by a simplified proof of the \emph{isoperimetric inequality for R\'enyi entropy powers} in~\cite{MR3200617}, based on based on the linear growth of the functional $\mathsf F$, that is also stated in \cite[Theorem~4.4]{carrillo2014renyi} and emphasize the fact that the method is limited to exponents $p\ge1-\frac1d$. Using this inequality, J.A.~Carrillo and G.~Toscani have been able to establish improved rates of convergence in the non-asymptotic regime of the solutions to~\eqref{poro} in \cite[Theorem~5.1]{carrillo2014renyi}. Since R\'enyi entropy powers are equivalent to \emph{relative entropies relative to best matching Barenblatt solutions} under appropriate conditions on the second moments as shown in \cite{1004,MR3103175}, improved rates have to be related with results obtained in rescaled variables in~\cite{MR3103175}, where the scales were defined in terms of second moments. Equivalently, the \emph{improved entropy-entropy production inequality} of \cite[Theorem~1]{MR3103175} is similar in nature to the \emph{isoperimetric inequality for R\'enyi entropy powers}. As we shall observe in this paper, the \emph{isoperimetric inequality for R\'enyi entropy powers} is in fact a \emph{Gagliardo-Nirenberg inequality} in scale invariant form, which degenerates into Sobolev's inequality in the limit case $p=1-\frac1d$ and explains why it cannot hold true for $p<1-\frac1d$. The \emph{entropy-entropy production inequality} of \cite{MR1940370} is also a Gagliardo-Nirenberg inequality, but not in scale invariant form and it has recently been established in \cite{dolbeault:hal-01081098} that the difference between the scale invariant form and the non scale invariant form is enough to account for improved rates of convergence in Fokker-Planck type equations, \emph{i.e.}, in nonlinear diffusions after a convenient rescaling.

Theorem~\ref{Thm:T2} is devoted to a rather simple observation on second moments, which is however at the core of our paper. It is based on the linear growth of the functional $\mathsf G$. By introducing the \emph{relative entropy with respect to the best matching Barenblatt function}, we establish the result of Theorem~\ref{Thm:T3} on delays. This result has been proved in \cite{2014arXiv1408.6781D} by a much more complicated method, when $p<1$. Here we simply rely on the convexity or concavity of $t\mapsto\mathsf G(t)$, depending whether $p<1$ or $p>1$, and both cases are covered. This analysis also suggests a method to obtain estimates of the delays, that is investigated in Section~\ref{Sec:Delays}.

\medskip Let us conclude this introduction with a brief review of the literature. For considerations on second moment methods in fast diffusion and porous medium equations, we refer to \cite{MR2133441,MR2328935,2014arXiv1408.6781D} and references therein. Sharp Ga\-gliardo-Nirenberg-Sobolev inequalities have been studied in~\cite{MR1940370,MR3103175,MR3200617,dolbeault:hal-01081098} from the point of view of the rates of convergence of the solutions to~\eqref{poro} in the intermediate asymptotics regime, and also for obtaining improved convergence rates in the initial regime. A counterpart of such improved rates is the notion of \emph{delay} which was established in~\cite{2014arXiv1408.6781D} in the fast diffusion regime and will be recovered as a very simple consequence of moment estimates in Theorem~\ref{Thm:T3} and extended to the porous medium case. 

Theorems~\ref{Thm:T1} and~\ref{Thm:T2} follow the same line of thought: compute the time evolution of a functional which grows linearly when evaluated in the case of Barenblatt functions and has some concavity or convexity property otherwise. As a main consequence, we provide a proof of some Gagliardo-Nirenberg-Sobolev inequalities which goes along the lines of \cite{MR1777035,MR1940370,MR2745814} on the one hand, of \cite{MR3200617,carrillo2014renyi,MR3255069} on the other hand, and makes a synthetic link between the two approaches. The first approach is inspired by the entropy functional introduced by J.~Ralston and W.I.~Newman in \cite{MR760591,MR760592}, also known in the literature as the \emph{Tsallis entropy}, while the second one is more related with \emph{R\'enyi entropies} and connected with information theory inspired by \cite{MR823597,MR1768665}. The reader interested in further details is invited to refer to \cite{MR3255069} and references therein for a more detailed account.

\section{Notations, main results and consequences}\label{Sec:Main}

The \emph{entropy} is defined by
\[
\mathsf E:=\ird{u^p}
\]
and the \emph{Fisher information} by
\[
\mathsf I:=\ird{u\,|\nabla v|^2}\quad\mbox{with}\quad v=\frac p{p-1}\,u^{p-1}\,.
\]
If $u$ solves~\eqref{poro}, then
\[
\mathsf E'=(1-p)\,\mathsf I\,.
\]
To compute $\mathsf I'$, we will use the fact that
\be{Eqn:v}
\frac{\partial v}{\partial t}=(p-1)\,v\,\Delta v+|\nabla v|^2
\ee
and get that
\[
\mathsf F:=\mathsf E^\sigma
\]
has a linear growth asymptotically as $t\to+\infty$ if
\[\label{sigma}
\sigma=\frac\mu{d\,(1-p)}=1+\frac2{1-p}\,\(\frac 1d+p-1\)=\frac 2d\,\frac1{1-p}-1\,.
\]
This definition is the same as the one of Section~\ref{Sec:Intro}. The growth is exactly linear in case of Barenblatt profiles, so that $\mathsf E_\star^{\sigma-1}\,\mathsf I_\star$ is independent of $t$ if we denote by $\mathsf E_\star$, $\mathsf I_\star$, $\mathsf F_\star$ \emph{etc.} the entropy, the Fisher information, \emph{etc.} of these Barenblatt profiles.
\begin{thm}\label{Thm:T1} Assume that $p\ge1-\frac1d$ if $d>1$ and $p>0$ if $d=1$. With the above notations, $t\mapsto\mathsf F(t)$ is increasing, $(1-p)\,\mathsf F''(t)\le0$ and
\[
\lim_{t\to+\infty}\frac1t\,\mathsf F(t)=(1-p)\,\sigma\,\lim_{t\to+\infty}\mathsf E^{\sigma-1}\,\mathsf I=(1-p)\,\sigma\,\mathsf E_\star^{\sigma-1}\,\mathsf I_\star\,.
\]
\end{thm}
This result has been established in \cite{MR3200617}. In this paper we give a slightly simpler proof. The result of Theorem~\ref{Thm:T1} amounts to state that $t\mapsto\mathsf F(t)$ is concave if $p<1$ and convex if $p>1$, with an asymptotic slope given by $(1-p)\,\sigma\,\mathsf E_\star^{\sigma-1}\,\mathsf I_\star$. Moreover, we observe that the inequality
\be{GN-Barenblatt}
\mathsf E^{\sigma-1}\,\mathsf I\ge\mathsf E_\star^{\sigma-1}\,\mathsf I_\star
\ee
is equivalent to one of the two following Gagliardo-Nirenberg inequalities:
\begin{description}
\item[(i)] If $1-\frac1d\le p<1$, then~\eqref{GN-Barenblatt} is equivalent to
\be{GN1}
\nrm{\nabla w}2^\theta\,\nrm w{q+1}^{1-\theta}\ge\mathsf C_{\rm{GN}}\,\nrm w{2q}
\ee
where
\[
\theta=\frac dq\,\frac{q-1}{d+2-q\,(d-2)}\,,\quad1<q\le\frac d{d-2}
\]
and equality with optimal constant $\mathsf C_{\rm{GN}}$ is achieved by $\mathcal B_\star^{p-1/2}$. 
\item[(ii)] If $p>1$, then~\eqref{GN-Barenblatt} is equivalent to
\be{GN2}
\nrm{\nabla w}2^\theta\(\ird{w^{2q}}\)^\frac{1-\theta}{2q}\ge\mathsf C_{\rm{GN}}\,\nrm w{q+1}
\ee
where
\[
\theta=\frac dq\,\frac{q-1}{d+2-q\,(d-2)}
\]
and $q$ takes any value in $(0,1)$.
\end{description}
In both cases, we relate Gagliardo-Nirenberg inequalities and~\eqref{GN-Barenblatt} by
\[
u^{p-1/2}=\frac w{\nrm w{2q}}\quad\mbox{with}\quad q=\frac1{2\,p-1}\,.
\]
These considerations show that $p=1-\frac1d$ correspond to Sobolev's inequality: $2q=2\,d/(d-2)$ and $\theta=1$ if $d\ge3$. This case is therefore the threshold case for the validity of the method. Details will be given in Section~\ref{Sec:GN}.

\medskip Using the moment
\[
\Theta:=\frac 1d\ird{|x|^2\,u}\,,
\]
we can introduce the functional
\[
\mathsf G:=\Theta^{1-\eta/2}\quad\mbox{with}\quad\eta=d\,(1-p)\,.
\]
As a function of $t$, $\mathsf G$ also has a linear growth in case of Barenblatt profiles. Again this definition is compatible with the one of Section~\ref{Sec:Intro} since
\[
\frac\mu2=1-\frac\eta2\,.
\]
In the general case, we get that
\[
\mathsf G'=\mu\,\mathsf H\,,
\]
where we define the \emph{R\'enyi entropy power} functional $\mathsf H$ by
\[
\mathsf H:=\Theta^{-\eta/2}\,\mathsf E=\Theta^{\frac d2\,(p-1)}\,\mathsf E
\]
and observe that the corresponding functional $\mathsf H_\star$ for Barenblatt profiles is independent of $t$.
\begin{thm}\label{Thm:T2} Assume that $p\ge1-\frac2d$. With the above notations, $t\mapsto\mathsf G(t)$ is increasing, $(1-p)\,\mathsf H'(t)\ge0$ and
\[
\lim_{t\to+\infty}\frac1t\,\mathsf G(t)=\(1-\frac\eta2\)\lim_{t\to+\infty}\mathsf H=\(1-\frac\eta2\)\mathsf H_\star\,.
\]
\end{thm}
The function $t\mapsto\mathsf G(t)$ is convex if $p<1$ and concave if $p>1$, with an asymptotic slope given by $(1-\eta/2)\,\mathsf H_\star$. Details will be given in Section~\ref{Sec:Moment}.

\medskip We can also consider the \emph{relative entropy with respect to the best matching Barenblatt function} defined as
\[
\mathcal F[u]:=\inf_{s>0}\mathcal F[u\,|\,\mathcal U_\star^s]
\]
where $\mathcal U_\star^s(x)=\mathcal U_\star(s,x)=s^{-d/\mu}\,\mathcal U_\star^1(s^{-1/\mu}\,x)$ is defined by~\eqref{ba-self}. Here the variable~$s$ plays the role of a scaling parameter, and the \emph{relative entropy} with respect to a given function $\mathcal U$ is defined by
\[
\mathcal F[u\,|\,\mathcal U]:=\frac1{p-1}\ird{\Big[u^p-\mathcal U^p-p\,\mathcal U^{p-1}\,\big(u-\mathcal U\big)\Big]}\,.
\]
A key observation in \cite{1004} is the fact that $\mathcal F[u]=\mathcal F[u\,|\,\mathcal U_\star^s]$ is achieved by a unique Barenblatt solution which satisfies
\[
\ird{|x|^2\,u}=\ird{|x|^2\,\mathcal U_\star^s}\,.
\]
The proof of this fact is a straightforward computation: as a functional of~$\mathcal U$, $\mathcal F[u\,|\,\mathcal U]$ is concave and has at most one maximum point. Hence if we write that $|\mathcal U_\star^s(x)|^{p/(p-1)}=a(s)+b(s)\,|x|^2$, then
\[
\frac d{ds}\mathcal F[u\,|\,\mathcal U_\star^s]=\frac p{p-1}\ird{\(a'(s)+b'(s)\,|x|^2\)\(u-\mathcal U_\star^s\)}\,.
\]
Since $\ird u=\ird{\mathcal U_\star^s}=1$, the term proportional to $a'$ does not contribute and the other one vanishes if and only if the moments are equal. 

After undoing the change of variables, we get that
\be{Eqn:s}
s=\(\frac\Theta{\Theta_\star}\)^\frac\mu2\,.
\ee
With this choice of $s$, $\mathcal U_\star^s$ is the \emph{best matching Barenblatt} function in the sense that this Barenblatt function minimizes the relative entropy $\mathcal F[u\,|\,\mathcal U_\star^s]$ among all Barenblatt functions $(\mathcal U_\star^s)_{s>0}$.

If $u$ is a solution to~\eqref{poro}, for any $t\ge0$, we can define $s$ as a function of $t$ and consider the \emph{delay} which is defined as
\[
\tau(t):=\(\frac{\Theta(t)}{\Theta_\star}\)^\frac\mu2-t\,.
\]
The main result of this paper is that $t\mapsto\tau(t)$ is monotone.
\begin{thm}\label{Thm:T3} Assume that $p\ge1-\frac1d$ and $p\neq1$. With the above notations, the best matching Barenblatt function of a solution $u$ to~\eqref{poro} satisfying~\eqref{second} is $(t,x)\mapsto\mathcal U_\star(t+\tau(t),x)$ and the function $t\mapsto\tau(t)$ is nondecreasing if $p>1$ and nonincreasing if $1-\frac1d\le p<1$.\end{thm}
With $s$ given by~\eqref{Eqn:s}, we notice that the \emph{relative entropy with respect to the best matching Barenblatt function} is given by
\[
\mathcal F[u\,|\,\mathcal U_\star^s]=\frac1{p-1}\ird{\Big[u^p-(\mathcal U_\star^s)^p\Big]}
\]
and it is related with the the \emph{R\'enyi entropy power} functional by
\[
\mathsf H-\mathsf H_\star=\Theta^{\frac d2\,(p-1)}\,\mathsf E-\Theta_\star^{\frac d2\,(p-1)}\,\mathsf E_\star=(p-1)\,\Theta^{\frac d2\,(p-1)}\,\mathcal F[u\,|\,\mathcal U_\star^s]\,.
\]
\begin{figure}[ht]\begin{center}
\includegraphics[width=11cm]{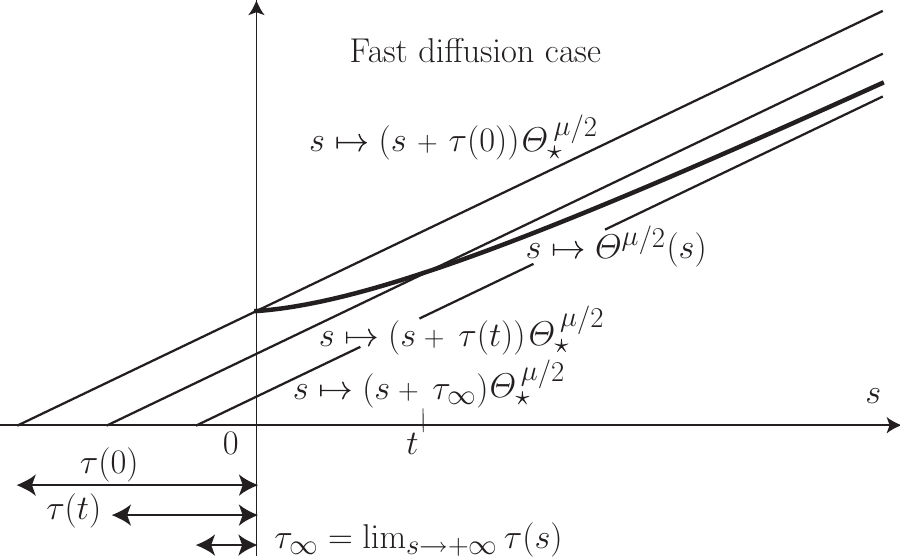}
\caption{\label{F1} In the fast diffusion case, the (bold) curve $s\mapsto\mathsf G(s)=\Theta^{\mu/2}(s)$ is convex, increasing, and its asymptote is $s\mapsto(s+\tau_\infty)\,\Theta_\star^{\mu/2}$. At any time $t\ge0$, we observe that it crosses the line $s\mapsto(s+\tau(t))\,\Theta_\star^{\mu/2}=(s-t+1)\,\Theta^{\mu/2}(t)$ transversally, so that $t\mapsto\tau(t)$ is monotone decreasing, unless the solution is itself a Barenblatt solution up to a time shift.}
\end{center}\end{figure}
\begin{figure}[ht]\begin{center}
\includegraphics[width=11cm]{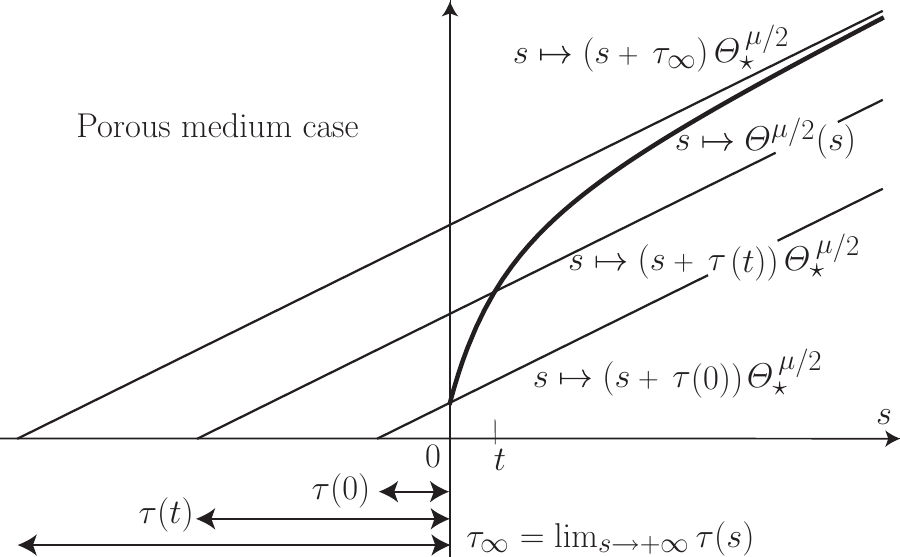}
\caption{\label{F2} In the porous medium case, the curve $s\mapsto\Theta^{\mu/2}(s)$ is still increasing but concave, so that $t\mapsto\tau(t)$ is monotone increasing, unless the solution is itself a Barenblatt solution up to a time shift.}
\end{center}\end{figure}
See Figs.~\ref{F1}-\ref{F2}. An estimate of the delay $t\mapsto\tau(t)$ will be given in Section~\ref{Sec:Delays}.

\section{A direct proof of Gagliardo-Nirenberg inequalities}\label{Sec:GN}

This section is devoted to the proof of Theorem~\ref{Thm:T1} and also to a proof of Gagliardo-Nirenberg inequalities~\eqref{GN1} and \eqref{GN2}. We provide a slightly simplified proof of Lemma~\ref{Lem:DerivFisher} below, compared to the existing literature: see \cite{MR3200617} and references therein. Some of the key computations will be reused in Section~\ref{Sec:Delays}.

\subsection{A preliminary computation}

According to~\cite[Appendix~B]{MR3200617}, we have the following result.
\begin{lem}\label{Lem:DerivFisher} If $u$ solves~\eqref{poro} with initial datum $u(x,t=0)=u_0(x)\ge0$ such that $\ird{u_0}=1$ and $\frac1d\ird{|x|^2\,u_0}<+\infty$, then $v=\frac p{p-1}\,u^{p-1}$ solves~\eqref{Eqn:v} and
\be{BLW}
\mathsf I'=\frac d{dt}\ird{u\,|\nabla v|^2}=-\,2\ird{u^p\,\Big(\|\mathrm D^2v\|^2+(p-1)\,(\Delta v)^2\Big)}\,.
\ee
\end{lem}
\newpage\begin{proof} Let us give a simplified proof of this result. Using~\eqref{poro} and~\eqref{Eqn:v}, we can compute
\begin{eqnarray*}
&&\hspace*{-18pt}\frac d{dt}\ird{u\,|\nabla v|^2}\\
&=&\ird{\frac{\partial u}{\partial t}\,|\nabla v|^2}+\,2\ird{u\,\nabla v\cdot\nabla\frac{\partial v}{\partial t}}\\
&=&\ird{\Delta(u^p)\,|\nabla v|^2}+\,2\ird{u\,\nabla v\cdot\nabla\Big((p-1)\,v\,\Delta v+|\nabla v|^2\Big)}\\
&=&\ird{u^p\,\Delta|\nabla v|^2}\\
&&+\,2\,(p-1)\ird{u\,v\,\nabla v\cdot\nabla\Delta v}+\,2\,(p-1)\ird{u\,\nabla v\cdot\nabla v\,\Delta v}\\
&&+\,2\ird{u\,\nabla v\cdot\nabla|\nabla v|^2}\\
&=&-\ird{u^p\,\Delta|\nabla v|^2}\\
&&+\,2\,(p-1)\ird{u\,v\,\nabla v\cdot\nabla\Delta v}+\,2\,(p-1)\ird{u\,\nabla v\cdot\nabla v\,\Delta v}
\end{eqnarray*}
where the last line is given by an integration by parts:
\[
\ird{u\,\nabla v\cdot\nabla|\nabla v|^2}=\ird{\nabla(u^p)\cdot\nabla|\nabla v|^2}=-\ird{u^p\,\Delta|\nabla v|^2}\,.
\]
1) Using the elementary identity
\[
\frac12\,\Delta\,|\nabla v|^2=\|\mathrm D^2v\|^2+\nabla v\cdot\nabla\Delta v\,,
\]
we get that
\[
\ird{u^p\,\Delta|\nabla v|^2}=2\ird{u^p\,\|\mathrm D^2v\|^2}+2\ird{u^p\,\nabla v\cdot\nabla\Delta v}\,.
\]
2) Since $u\,\nabla v=\nabla(u^p)$, an integration by parts gives
\begin{multline*}
\ird{u\,\nabla v\cdot\nabla v\,\Delta v}=\ird{\nabla(u^p)\cdot\nabla v\,\Delta v}\\
=-\ird{u^p\,(\Delta v)^2}-\ird{u^p\,\nabla v\cdot\nabla\Delta v}
\end{multline*}
and with $u\,v=\frac p{p-1}\,u^p$ we find that
\begin{multline*}
2\,(p-1)\ird{u\,v\,\nabla v\cdot\nabla\Delta v}+\,2\,(p-1)\ird{u\,\nabla v\cdot\nabla v\,\Delta v}\\
=-\,2\,(p-1)\ird{u^p\,(\Delta v)^2}+2\ird{u^p\,\nabla v\cdot\nabla\Delta v}\,.
\end{multline*}
Collecting terms establishes~\eqref{BLW}.\qed\end{proof}

\subsection{The fast diffusion case}
Recall that $\mathsf E=\ird{u^p}$ satisfies $\mathsf E'=(1-p)\,\mathsf I$ with $\mathsf I=\ird{u\,|\nabla v|^2}$ and $v=\frac p{p-1}\,u^{p-1}$. Since
\[
\|\mathrm D^2v\|^2=\frac 1d\,(\Delta v)^2+\left\|\,\mathrm D^2v-\frac 1d\,\Delta v\,\mathrm{Id}\,\right\|^2
\]
by Lemma~\ref{Lem:DerivFisher}, we find that $\mathsf F''=\(\mathsf E^\sigma\)''$ can be computed as
\begin{multline*}
\frac 1{\sigma\,(1-p)}\,\mathsf E^{2-\sigma}\,\(\mathsf E^\sigma\)''=(1-p)\,(\sigma-1)\(\ird{u\,|\nabla v|^2}\)^2\\
\hspace*{3cm}-\,2\,\(\frac 1d+p-1\)\ird{u^p}\ird{u^p\,(\Delta v)^2}\\
-\,2\ird{u^p}\ird{u^p\,\left\|\,\mathrm D^2v-\frac 1d\,\Delta v\,\mathrm{Id}\,\right\|^2}\,.
\end{multline*}
Using $u\,\nabla v=\nabla(u^p)$, we know that
\[
\ird{u\,|\nabla v|^2}=\ird{\nabla(u^p)\cdot\nabla v}=-\ird{u^p\,\Delta v}
\]
and by the Cauchy-Schwarz inequality,
\[
\(\ird{u\,|\nabla v|^2}\)^2\le\ird{u^p}\ird{u^p\,(\Delta v)^2}\,.
\]
With the choice
\[
\sigma=1+\frac2{1-p}\,\(\frac 1d+p-1\)=\frac 2d\,\frac1{1-p}-1\,,
\]
we get that
\[
\frac{\mathsf E^{2-\sigma}\,\(\mathsf E^\sigma\)''}{\sigma\,(1-p)^2}=:-\,\mathsf R[u]
\]
where the remainder terms have been collected as the sum of two squares:
\begin{multline*}
\frac{\mathsf R[u]}{\,\ird{u^p}}=(\sigma-1)\ird{u^p\,\left|\Delta v-\frac{\ird{u\,|\nabla v|^2}}{\ird{u^p}}\right|^2}\\
+\,\frac2{1-p}\ird{u^p\,\left\|\,\mathrm D^2v-\frac 1d\,\Delta v\,\mathrm{Id}\,\right\|^2}\,.
\end{multline*}

Hence we know that $\mathsf F'=\(\mathsf E^\sigma\)'$ is nonincreasing, that is,
\[
\frac1{1-p}\mathsf E^{\sigma-1}\,\mathsf E'=\(\ird{u^p}\)^{\sigma-1}\ird{u\,|\nabla v|^2}:=\mathsf J
\]
is nonincreasing. Since $\mathsf J$ is invariant under scalings as a functional of $u$, this means that
\[
\lim_{t\to\infty}\mathsf J=\mathsf J_\star
\]
where $\mathcal B_\star$ is the Barenblatt function such that $\ird{\mathcal B_\star}=1$ and $\mathsf J_\star$ the corresponding value of the functional $\mathsf J$. Written for the initial datum $u_0=u$, we have shown that
\[
\mathsf J\ge\mathsf J_\star\,.
\]
For any smooth and compactly supported function $w$, if we write $u^{p-1/2}=w/\nrm w{2q}$ with $q=1/(2\,p-1)$, then the inequality amounts to the Gagliar\-do-Nirenberg inequality~\eqref{GN1} and equality is achieved by $\mathcal B_\star^{p-1/2}$. More precisely we have shown the following result.
\begin{prop}\label{Prop:GN} Assume that $1<q<\frac d{d-2}$ if $d\ge3$ and $q>1$ if $d=1$ or $d=2$. With the above notations we have
\begin{multline*}
\frac{4\,p^2}{(p-1)^2\,(2\,p-1)^2}\,\Big(\mathsf J-\mathsf J_\star\Big)=\frac{\nrm{\nabla w}2^2\,\nrm w{q+1}^{2\,(1-\theta)/\theta}}{\nrm w{2q}^{2/\theta}}-\mathsf C_{\rm{GN}}^{2/\theta}\\
=(1-p)\int_0^\infty\mathsf R[u(t,\cdot)]\,dt\ge0\,,
\end{multline*}
where $u(t,\cdot)$ denotes the solution to~\eqref{poro} with initial datum $u$.\end{prop}
Hence we have shown that, as a function of $t$, $\mathsf F$ is concave increasing and we have identified its asymptotic slope, which is given by the optimal constant in the Gagliardo-Nirenberg inequality
\[
\lim_{t\to+\infty}\mathsf J(t)=\mathsf J_\star=\mathsf C_{\rm{GN}}^{2/\theta}\,.
\]

\subsection{The porous medium case}

The computations are the same. With the same definition~\eqref{sigma} as in the fast diffusion case, $\sigma$ is negative, $t\mapsto\mathsf F$ is convex increasing and the limit of its derivative is achieved among Barenblatt functions. The Gagliardo-Nirenberg inequality now takes the form~\eqref{GN2}. Again equality is achieved by $\mathcal B_\star^{p-1/2}$. Details are left to the reader.

\section{The second moment and the R\'enyi entropy power functional}\label{Sec:Moment}

This section is devoted to the proof of Theorem~\ref{Thm:T2}. Let us consider the subsequent time derivatives of the functional
\[
\mathsf G:=\Theta^{1-\frac\eta2}\quad\mbox{with}\quad\eta=d\,(1-p)=2-\mu\,.
\]
It is straightforward to check that
\[
\mathsf G'=\mu\,\mathsf H\quad\mbox{with}\quad\mathsf H:=\Theta^{-\frac\eta2}\,\mathsf E
\]
where the \emph{R\'enyi entropy power} functional is defined by
\[
\mathsf H:=\Theta^{-\frac\eta2}\,\mathsf E\,.
\]
We recall that $\mathsf E'=(1-p)\,\mathsf I$. It is straightforward to check that
\[
\frac{\mathsf H'}{1-p}=\Theta^{-1-\frac\eta2}\(\Theta\,\mathsf I- d\,\mathsf E^2\)=\frac{d\,\mathsf E^2}{\Theta^{\frac\eta2+1}}\,(\mathsf q-1)\quad\mbox{with}\quad\mathsf q:=\frac{\Theta\,\mathsf I}{d\,\mathsf E^2}\ge1
\]
because
\begin{multline*}
d\,\mathsf E^2=\frac 1d\(-\ird{x\cdot\nabla(u^p)}\)^2=\frac 1d\(\ird{x\cdot u\,\nabla v}\)^2\\
\le\frac 1d\ird{u\,|x|^2}\ird{u\,|\nabla v|^2}=\Theta\,\mathsf I
\end{multline*}
by the Cauchy-Schwarz inequality. This proves that $(1-p)\,\mathsf H$ is monotone increasing and from the theory in \cite{MR1940370}, we know that its limit is given by the source-type Barenblatt solutions $\mathcal U_{p,t}$ defined by~\eqref{ba-self}. Since $\mathsf H$ is also scale invariant, its limit is in fact given by the value of the functional for $\mathcal B_\star$, that~is
\[
\lim_{t\to+\infty}\mathsf H(t)=\lim_{t\to+\infty}\Theta(t)^{\frac d2\,(p-1)}\,\mathsf E(t)=\Theta_\star^{\frac d2\,(p-1)}\,\mathsf E_\star=:\mathsf H_\star
\]
and we get
\be{pri}
\begin{array}{ll}
\Theta(t)^{\frac d2\,(p-1)}\,\mathsf E(t)\le\mathsf H_\star\quad&\mbox{if}\quad p<1\,,\\[5pt]
\Theta(t)^{\frac d2\,(p-1)}\,\mathsf E(t)\ge\mathsf H_\star\quad&\mbox{if}\quad p>1\,.
\end{array}
\ee
Taking the logarithm on both sides of inequality~\eqref{pri}, and considering that $p<1$, we obtain the equivalent inequality
\[\label{pr2}
\frac1{1-p}\log\ird{u^p}-\frac d2\,\log\(\frac1d\ird{|x|^2\,u}\)\ge\frac{\log\mathsf H_\star}{1-p}=\frac{\log\mathsf E_\star}{1-p}-\frac d2\,\log\Theta_\star
\]
in which we can recognize the well-known inequality for R\'enyi entropies obtained in~\cite{CHV,LYZ}. After multiplying by $(1-p)$ and taking the exponential, we realize that this is also equivalent to the inequality
\[
\mathcal F[u\,|\,\mathcal U_\star^1]\ge0
\]
which is a well-known consequence of Jensen's inequality: see for instance \cite{MR1940370}.

\section{Delays}\label{Sec:Delays}

This section is devoted to delays. At any time $t\ge0$, we consider the best matching Barenblatt solution and the corresponding delay $\tau(t)$
\[
\mathcal U_\star^{s(t)}(x)=\mathcal U_\star(s(t),x)\,,\quad s(t)=t+\tau(t)
\]
according to the definitions of Section~\ref{Sec:Main}. After proving Theorem~\ref{Thm:T3}, we give an estimate of $\tau(t)$.

\medskip\noindent\emph{Proof of Theorem~\ref{Thm:T3}.} With $p\ge1-\frac1d$, we know that $\mu>0$. At any time $t\ge0$, let us consider the solution $t\mapsto u(t,\cdot)$ to~\eqref{poro}. The scale $s(t)$ of the best matching Barenblatt function is determined by~\eqref{Eqn:s}. By observing that $t\mapsto\Theta^{\mu/2}$ grows faster (resp. slower) than $t\mapsto t\,\Theta_\star^{\mu/2}$ if $p>1$ (resp.~if $p<1$), where $t\mapsto t\,\Theta_\star^{\mu/2}$ is the rate of growth corresponding to the self-similar Barenblatt function given by~\eqref{ba-self}, we get that $t\mapsto\tau(t)$ is nondecreasing (resp. nonincreasing) in the porous medium case $p>1$ (resp. fast diffusion case $p<1$). See Figs.~\ref{F1}-\ref{F2}. \qed

\medskip Next we study a quantitative estimate of delays, which relies on the commutation of the third derivative in $t$ of $\mathsf G$. This approach is parallel to the results in \cite{2014arXiv1410.2722T} in the linear case.

Let us recall that $\Theta'=2\,\mathsf E$, $\mathsf E'=(1-p)\,\mathsf I$ and $\mathsf J=\mathsf E^{\sigma-1}\,\mathsf I$ is such that
\[
0\ge\mathsf E^{2-\sigma}\,\mathsf J'=(1-p)\,(\sigma-1)\,\mathsf I^2+\mathsf E\,\mathsf I'
\]
so that 
\[
\mathsf I'\le-\,(1-p)\,(\sigma-1)\,\frac{\mathsf I^2}{\mathsf E}=-\,\frac2d\,(1-\eta)\,\frac{\mathsf I^2}{\mathsf E}\,.
\]
Hence
\be{H''}
\frac{\mathsf H''}{1-p}=\(\Theta^{-\frac\eta2}\,\mathsf I-\,d\,\Theta^{-\frac\eta2-1}\,\mathsf E^2\)'\le\frac{d\,\mathsf E^3}{\Theta^{\frac\eta2+2}}\,\Big(\eta+2-\,3\,\eta\,\mathsf q-\,2\,(1-\eta)\,\mathsf q^2\Big)
\ee
according to the computations of Section~\ref{Sec:Moment}. We observe that $\mathsf H''\le0$ if $p<1$ and $\mathsf H''\ge0$ if $p>1$.
\begin{thm}\label{Thm:T3bis} Under the assumptions of Theorem~\ref{Thm:T3}, if $p>1-\frac1d$ and $p\neq1$, then the delay satisfies
\[
\lim_{t\to+\infty}|\tau(t)-\tau(0)|\ge|1-p|\,\frac{\Theta(0)^{1-\frac d2(1-p)}}{2\,\mathsf H_\star}\,\frac{\big(\mathsf H_\star-\mathsf H(0)\big)^2}{\Theta(0)\,\mathsf I(0)-d\,\mathsf E(0)^2}
\]
\end{thm}
\begin{proof} Assume first that $p<1$ and recall that for any $t\ge0$
\[
\mathsf G'(t)\le\lim_{t\to+\infty}\mathsf G'(t)=\(1-\frac\eta2\)\mathsf H_\star=:\mathsf G_\star'\,.
\]
Since $\mathsf G''$ is nonincreasing, we have the estimate
\[
\mathsf G(t)\le\mathsf G(0)+\mathsf G'(0)\,t+\tfrac12\,\mathsf G''(0)\,t^2\quad\forall\,t\ge0
\]
so that
\[
\mathsf G(0)+\mathsf G_\star'\,t-\mathsf G(t)\ge\(\mathsf G_\star'-\mathsf G'(0)\)t-\tfrac12\,\mathsf G''(0)\,t^2
\]
is maximal for $t=t_\star:=(\mathsf G_\star'-\mathsf G'(0))/\mathsf G''(0)$. As a consequence, since $\mathsf G(0)=\mathsf G_\star'\,\tau(0)$ and $\mathsf G(t_\star)=\mathsf G_\star'\,\big(t_\star+\tau(t_\star)\big)$, we get that
\[
\mathsf G_\star'\,\tau(0)-\mathsf G_\star'\,\tau(t_\star)=\mathsf G(0)+\mathsf G_\star'\,t_\star-\mathsf G_\star'\,\big(t_\star+\tau(t_\star)\big)\ge\frac{\big(\mathsf G_\star'-\mathsf G'(0)\big)^2}{2\,\mathsf G''(0)}\,,
\]
that is
\[
\tau(0)-\tau(t_\star)\ge\frac{\big(\mathsf G_\star'-\mathsf G'(0)\big)^2}{2\,\mathsf G_\star'\,\mathsf G''(0)}=(1-p)\,\frac{\Theta(0)^{1+\frac\eta2}}{2\,\mathsf H_\star}\,\frac{\big(\mathsf H_\star-\mathsf H(0)\big)^2}{\Theta(0)\,\mathsf I(0)-d\,\mathsf E(0)^2}\,.
\]
We conclude by observing that $t\mapsto\tau(0)-\tau(t)$ is nondecreasing.

Estimates for $p>1$ are very similar, up to signs. \end{proof}

Further estimates can be obtained easily. Let us illustrate our new approach by a result on
\be{q}
\mathsf q:=\frac{\Theta\,\mathsf I}{d\,\mathsf E^2}=1+\frac{\Theta^{\frac\eta2+1}}{d\,\mathsf E^2}\,\frac{\mathsf H'}{1-p}
\ee
in the fast diffusion case. We denote by $\mathsf q_0$, $\Theta_0$,.. the initial values of $\mathsf q$, $\Theta$, \emph{etc.}
\begin{prop}\label{Thm:T4} Assume that $1-\frac1d\le p<1$. Then for any $t\ge0$, we have
\[
\mathsf q(t)\le\frac{\mathsf q_0\,\Theta(t)}{\mathsf q_0\,\Theta(t)-(\mathsf q_0-1)\,\Theta_0}=:\bar{\mathsf q}(t)
\]
and
\[
\tau(t)\le\tau_0\,\exp\left[\int_0^t\frac{ds}{s+\frac{\Theta_0}{\mu\,\mathsf E_0}-\frac\eta\mu\int_0^s(\bar{\mathsf q}-1)\,ds}\right]-t\,.
\]
\end{prop}
\begin{proof} Using~\eqref{H''} and~\eqref{q}, we obtain
\[
\mathsf q'\le2\,\mathsf q\,(1-\mathsf q)\,\frac{\mathsf E}\Theta=\mathsf q\,(1-\mathsf q)\,\frac{\Theta'}\Theta\,.
\]
We can integrate and get the first estimate.

To obtain the integral estimate for $\tau$, we compute
\[
\frac{\mathsf G'}{\mathsf G}=\mu\,\frac{\mathsf E}\Theta
\]
and
\[
\frac1{1-p}\,\frac{\mathsf G''}{\mathsf G'}=\frac1{1-p}\,\frac{\mathsf H'}{\mathsf H}=d\,\frac{\mathsf E}\Theta\,(\mathsf q-1)\le\frac d\mu\,\frac{\mathsf G'}{\mathsf G}\,(\bar{\mathsf q}-1)\,,
\]
so that
\[
1-\(\frac{\mathsf G}{\mathsf G'}\)'=\frac{\mathsf G\,\mathsf G''}{(\mathsf G')^2}\le\frac\eta\mu\,(\bar{\mathsf q}-1)
\]
and finally
\[
\frac{\mathsf G}{\mathsf G'}\ge\frac{\Theta_0}{\mu\,\mathsf E_0}+t-\frac\eta\mu\int_0^t(\bar{\mathsf q}(s)-1)\,ds\,.
\]
Using the fact that $\mathsf G(t)=\mathsf G_\star'\,\big(t+\tau(t)\big)$, we conclude after one more integration with respect to $t$ and get the second estimate.\qed\end{proof}

Let us conclude this paper by a few remarks.
\begin{enumerate}
\item[(i)] The quantity $\mathsf q-1$ is a measure of the distance to the set of the Barenblatt profiles when $p>1-\frac 1d$: see \cite{dolbeault:hal-01081098} for more details.
\item[(ii)] Improved rates of decay for $\mathsf E$ and, as a consequence, improved asymptotics for $\mathsf F$ and $\mathsf G$ can be achieved by considering the estimates found in \cite{carrillo2014renyi}.
\item[(iii)] Alternatively, improved functional inequalities as in \cite{dolbeault:hal-01081098} can be used directly to get improved asymptotics for $\mathsf F$ and $\mathsf G$.
\end{enumerate}

\begin{acknowledgements} This work has been partially supported by the projects \emph{STAB}, \emph{No\-NAP} and \emph{Kibord} (J.D.) of the French National Research Agency (ANR), and by MIUR project ``Optimal mass transportation, geometrical and functional inequalities with applications'' (G.T.). J.D.~thanks the Department of Mathematics of the University of Pavia for inviting him and supporting his visit.
\par\medskip\noindent{\small\copyright\,2014 by the authors. This paper may be reproduced, in its entirety, for non-commercial purposes.}\end{acknowledgements}


\begin{thebibliography}{10}
\providecommand{\url}[1]{{#1}}
\providecommand{\urlprefix}{URL }
\expandafter\ifx\csname urlstyle\endcsname\relax
  \providecommand{\doi}[1]{DOI~\discretionary{}{}{}#1}\else
  \providecommand{\doi}{DOI~\discretionary{}{}{}\begingroup
  \urlstyle{rm}\Url}\fi

\bibitem{barenblatt1952some}
Barenblatt, G.I.: On some unsteady motions of a liquid and gas in a porous
  medium.
\newblock Prikl. Mat. Mekh \textbf{16}(1), 67--78 (1952)

\bibitem{MR2745814}
Carlen, E.A., Carrillo, J.A., Loss, M.: Hardy-{L}ittlewood-{S}obolev
  inequalities via fast diffusion flows.
\newblock Proc. Natl. Acad. Sci. USA \textbf{107}(46), 19,696--19,701 (2010).
\newblock \doi{10.1073/pnas.1008323107}.
\newblock \urlprefix\url{http:://dx.doi.org/10.1073/pnas.1008323107}

\bibitem{MR1777035}
Carrillo, J.A., Toscani, G.: Asymptotic {$\mathrm L^1$}-decay of solutions of
  the porous medium equation to self-similarity.
\newblock Indiana Univ. Math. J. \textbf{49}(1), 113--142 (2000).
\newblock \doi{10.1512/iumj.2000.49.1756}.
\newblock \urlprefix\url{http://dx.doi.org/10.1512/iumj.2000.49.1756}

\bibitem{carrillo2014renyi}
{Carrillo}, J.A., {Toscani}, G.: {Renyi entropy and improved equilibration
  rates to self-similarity for nonlinear diffusion equations}.
\newblock Nonlinearity \textbf{27}(12), 3159--3177 (2014).
\newblock \urlprefix\url{http://dx.doi.org/10.1088/0951-7715/27/12/3159}

\bibitem{MR2328935}
Carrillo, J.A., V{\'a}zquez, J.L.: Asymptotic complexity in filtration
  equations.
\newblock J. Evol. Equ. \textbf{7}(3), 471--495 (2007).
\newblock \doi{10.1007/s00028-006-0298-z}.
\newblock \urlprefix\url{http://dx.doi.org/10.1007/s00028-006-0298-z}

\bibitem{CHV}
Costa, J., Hero, A., Vignat, C.: On solutions to multivariate maximum
  α-entropy problems.
\newblock In: A.~Rangarajan, M.~Figueiredo, J.~Zerubia (eds.) Energy
  Minimization Methods in Computer Vision and Pattern Recognition,
  \emph{Lecture Notes in Computer Science}, vol. 2683, pp. 211--226. Springer
  Berlin Heidelberg (2003).
\newblock \doi{10.1007/978-3-540-45063-4_14}.
\newblock \urlprefix\url{http://dx.doi.org/10.1007/978-3-540-45063-4_14}

\bibitem{MR823597}
Costa, M.H.M.: A new entropy power inequality.
\newblock IEEE Trans. Inform. Theory \textbf{31}(6), 751--760 (1985).
\newblock \doi{10.1109/TIT.1985.1057105}.
\newblock \urlprefix\url{http:://dx.doi.org/10.1109/TIT.1985.1057105}

\bibitem{MR1940370}
Del~Pino, M., Dolbeault, J.: Best constants for {G}agliardo-{N}irenberg
  inequalities and applications to nonlinear diffusions.
\newblock J. Math. Pures Appl. (9) \textbf{81}(9), 847--875 (2002).
\newblock \doi{10.1016/S0021-7824(02)01266-7}.
\newblock \urlprefix\url{http://dx.doi.org/10.1016/S0021-7824(02)01266-7}

\bibitem{1004}
Dolbeault, J., Toscani, G.: Fast diffusion equations: matching large time
  asymptotics by relative entropy methods.
\newblock Kinetic and Related Models \textbf{4}(3), 701--716 (2011).
\newblock \urlprefix\url{http://dx.doi.org/10.3934/krm.2011.4.701}

\bibitem{MR3103175}
Dolbeault, J., Toscani, G.: Improved interpolation inequalities, relative
  entropy and fast diffusion equations.
\newblock Ann. Inst. H. Poincar\'e Anal. Non Lin\'eaire \textbf{30}(5),
  917--934 (2013).
\newblock \doi{10.1016/j.anihpc.2012.12.004}.
\newblock \urlprefix\url{http://dx.doi.org/10.1016/j.anihpc.2012.12.004}

\bibitem{dolbeault:hal-01081098}
Dolbeault, J., Toscani, G.: {Stability results for logarithmic Sobolev and
  Gagliardo-Niren\-berg inequalities} (2014).
\newblock \urlprefix\url{http:://hal.archives-ouvertes.fr/hal-01081098}.
\newblock Submitted

\bibitem{2014arXiv1408.6781D}
{Dolbeault}, J., {Toscani}, G.: {Best matching Barenblatt profiles are
  delayed}.
\newblock To appear in Journal of Physics A: Mathematical and Theoretical
  (2015).
\newblock \urlprefix\url{http://arxiv.org/abs/1408.6781}

\bibitem{MR2909910}
Fila, M., V{\'a}zquez, J.L., Winkler, M., Yanagida, E.: Rate of convergence to
  {B}arenblatt profiles for the fast diffusion equation.
\newblock Arch. Ration. Mech. Anal. \textbf{204}(2), 599--625 (2012).
\newblock \doi{10.1007/s00205-011-0486-z}.
\newblock \urlprefix\url{http:://dx.doi.org/10.1007/s00205-011-0486-z}

\bibitem{MR586735}
Friedman, A., Kamin, S.: The asymptotic behavior of gas in an {$n$}-dimensional
  porous medium.
\newblock Trans. Amer. Math. Soc. \textbf{262}(2), 551--563 (1980).
\newblock \doi{10.2307/1999846}.
\newblock \urlprefix\url{http://dx.doi.org/10.2307/1999846}

\bibitem{LYZ}
Lutwak, E., Yang, D., Zhang, G.: Cram\'er-{R}ao and moment-entropy inequalities
  for {R}enyi entropy and generalized {F}isher information.
\newblock IEEE Trans. Inform. Theory \textbf{51}(2), 473--478 (2005).
\newblock \doi{10.1109/TIT.2004.840871}.
\newblock \urlprefix\url{http://dx.doi.org/10.1109/TIT.2004.840871}

\bibitem{MR760591}
Newman, W.I.: A {L}yapunov functional for the evolution of solutions to the
  porous medium equation to self-similarity. {I}.
\newblock J. Math. Phys. \textbf{25}(10), 3120--3123 (1984)

\bibitem{MR760592}
Ralston, J.: A {L}yapunov functional for the evolution of solutions to the
  porous medium equation to self-similarity. {II}.
\newblock J. Math. Phys. \textbf{25}(10), 3124--3127 (1984)

\bibitem{MR3200617}
Savar{\'e}, G., Toscani, G.: The concavity of {R}\'enyi entropy power.
\newblock IEEE Trans. Inform. Theory \textbf{60}(5), 2687--2693 (2014).
\newblock \doi{10.1109/TIT.2014.2309341}.
\newblock \urlprefix\url{http://dx.doi.org/10.1109/TIT.2014.2309341}

\bibitem{MR2133441}
Toscani, G.: A central limit theorem for solutions of the porous medium
  equation.
\newblock J. Evol. Equ. \textbf{5}(2), 185--203 (2005).
\newblock \doi{10.1007/s00028-005-0183-1}.
\newblock \urlprefix\url{http://dx.doi.org/10.1007/s00028-005-0183-1}

\bibitem{2014arXiv1410.2722T}
{Toscani}, G.: {A concavity property for the reciprocal of Fisher information
  and its consequences on Costa's EPI}.
\newblock ArXiv e-prints  (2014).
\newblock \urlprefix\url{http://arxiv.org/abs/1410.2722}

\bibitem{MR3255069}
Toscani, G.: R\'enyi {E}ntropies and {N}onlinear {D}iffusion {E}quations.
\newblock Acta Appl. Math. \textbf{132}, 595--604 (2014).
\newblock \doi{10.1007/s10440-014-9933-9}.
\newblock \urlprefix\url{http://dx.doi.org/10.1007/s10440-014-9933-9}

\bibitem{MR2286292}
V{\'a}zquez, J.L.: The porous medium equation.
\newblock Oxford Mathematical Monographs. The Clarendon Press, Oxford
  University Press, Oxford (2007).
\newblock Mathematical theory

\bibitem{MR1768665}
Villani, C.: A short proof of the ``concavity of entropy power''.
\newblock IEEE Trans. Inform. Theory \textbf{46}(4), 1695--1696 (2000).
\newblock \doi{10.1109/18.850718}.
\newblock \urlprefix\url{http://dx.doi.org/10.1109/18.850718}

\bibitem{zel1950towards}
Zel'dovich, Y.B., Kompaneets, A.: Towards a theory of heat conduction with
  thermal conductivity depending on the temperature.
\newblock Collection of papers dedicated to 70th birthday of Academician AF
  Ioffe, Izd. Akad. Nauk SSSR, Moscow pp. 61--71 (1950)

\end{thebibliography}



\end{document}